\documentclass[11pt, a4paper]{amsart}
\usepackage{fullpage}
\usepackage{amssymb}
\usepackage{mathrsfs}
\usepackage{comment}
\usepackage{empheq}
\usepackage{amsmath, amsthm}
\usepackage{parskip}
\usepackage{graphicx}
\usepackage{realboxes}
\usepackage[colorlinks=true, allcolors=blue]{hyperref}
\usepackage{tikz,tikz-3dplot,tikz-cd,tkz-tab,tkz-euclide,pgf,pgfplots}
\usetikzlibrary{bending,quotes,angles,3d}
\usepackage{bbm}
\usepackage{url}
\usepackage{amsrefs}

\newcommand{\ol}{\overline}
\newcommand{\ul}{\underline}
\newcommand{\1}{\mathbbm{1}}
\newcommand{\llv}{\left\lvert}
\newcommand{\rrv}{\right\rvert}
\newcommand{\llb}{\left\lbrace}
\newcommand{\rrb}{\right\rbrace}

\DeclareMathOperator{\Tor}{Tor}
\DeclareMathOperator{\Ext}{Ext}
\DeclareMathOperator{\Hom}{Hom}
\DeclareMathOperator{\End}{End}

\DeclareMathOperator{\Rep}{Rep}
\newcommand{\cP}{\mathcal{P}}

\newtheorem{thm}{Theorem}[section]

\theoremstyle{plain}
\newtheorem{prop}[thm]{Proposition}
\newtheorem{lem}[thm]{Lemma}

\newtheorem*{theorem*}{Theorem}

\theoremstyle{definition}
\newtheorem{defn}[thm]{Definition}

\theoremstyle{remark}
\newtheorem{rem}[thm]{Remark}

\makeatletter
\@namedef{subjclassname@1991}{\textup{2020} Mathematics Subject Classification}
\makeatother

\title{Cohomology of coloured partition algebras}
\author{James Cranch}
\address{(James Cranch) School of Mathematical and Physical Sciences, University of Sheffield, Hounsfield Road, S3 7RH, UK}
\email{j.d.cranch@sheffield.ac.uk}
\author{Daniel Graves}
\address{(Daniel Graves) Lifelong Learning Centre, University of Leeds, Woodhouse, Leeds, LS2 9JT, UK}
\email{dan.graves92@gmail.com}
\date{}

\definecolor{bw1}{RGB}{230, 159,   0} 
\definecolor{bw2}{RGB}{ 86, 180, 233} 
\definecolor{bw3}{RGB}{  0, 158, 115} 
\definecolor{bw4}{RGB}{240, 228,  66} 
\definecolor{bw5}{RGB}{  0, 114, 178} 
\definecolor{bw6}{RGB}{213,  94,   0} 
\definecolor{bw7}{RGB}{204, 121, 167} 

\begin{document}



\begin{abstract}
Coloured partition algebras were introduced by Bloss and
exhibit a Schur-Weyl duality with certain complex reflection
groups. In this paper we show that these algebras exhibit homological
stability by demonstrating that their homology groups are stably
isomorphic to the homology groups of a wreath product, generalizing work of Boyd--Hepworth--Patzt and Boyde for the usual partition algebras.
\end{abstract}

\maketitle

\section{Introduction}

\emph{Coloured partition algebras} were first introduced by Bloss \cite{Bloss} as a generalization of the partition algebras introduced by Jones \cite{Jones} and Martin \cite{Martin-partition}.

For a commutative ring $k$, an element $\delta \in k$ and a group $G$, Bloss demonstrated that the $G$-coloured
partition algebra, $P_n(\delta,G)$ (recalled in Section \ref{CPA-sec}), exhibits a Schur-Weyl duality with the wreath product
groups $G\wr \Sigma_n$ \cite{Bloss}*{Theorem 6.6}. For a finite group
$G$, Mori \cite[Remark 4.25]{Mori} shows that the $G$-coloured
partition algebras arise as the endomorphism ring of a certain object
in a Deligne-type category (see \cites{Deligne,Knop1,Knop2,Mori}) built
from $\Rep(k[G])$, the category of $k[G]$-modules which are
finitely-generated and projective over $k$.

We study the (co)homology of coloured partition algebras. The algebra $P_n(\delta,G)$ can be equipped with an augmentation $P_n(\delta,G)\rightarrow k$. This is described in Subsection \ref{cohom-defn-subsec}. Loosely speaking it sends the invertible generators, which form the group $G\wr \Sigma_n$, to $1\in k$ and all other generators to $0\in k$. Let $\1$ denote the a copy of the ground ring $k$ upon which $P_n(\delta,G)$ acts via the augmentation. Following Benson \cite[Definition 2.4.4]{Benson1}, the \emph{homology} and \emph{cohomology} of $P_n(\delta,G)$ are defined as $\Tor_{\star}^{P_n(\delta,G)}(\1,\1)$ and $\Ext_{P_n(\delta,G)}^{\star}(\1,\1)$ respectively. Our main result is as follows, following immediately from Theorems \ref{iso-thm1} and \ref{iso-thm2} in the text.

\begin{thm}
\label{thm-A}
For $n\geqslant 2$, there exist natural isomorphisms of $k$-modules
\[ \Tor_{q}^{P_n(\delta,G)}(\mathbbm{1},\mathbbm{1}) \cong \Tor_q^{k[G\wr  \Sigma_n]}(\mathbbm{1},\mathbbm{1}) \quad \text{and} \quad \Ext_{P_n(\delta,G)}^q(\mathbbm{1},\mathbbm{1}) \cong \Ext_{k[G\wr \Sigma_n]}^q(\mathbbm{1},\mathbbm{1})\]
for $q\leqslant n$. These isomorphisms also hold for $n=1$ with $q=0$. Furthermore, if $\delta$ is invertible, these isomorphisms hold for all $n\geqslant 1$ and $q\geqslant 0$.
\end{thm}

If we take $G$ to be the trivial group and consider homology, this extends the known range (see \cite{Boyde2} by $1$. Furthermore, as a consequence of Theorem \ref{thm-A}, we see that the coloured partition algebras exhibit (co)homological stability. In particular, the inclusion map $P_{n-1}(\delta,G) \rightarrow P_n(\delta,G)$
 induces maps on homology and cohomology,
 \[\Tor_q^{P_{n-1}(\delta,G)}(\1 ,\1) \rightarrow \Tor_q^{P_{n}(\delta,G)}(\1 ,\1)\quad \text{and} \quad \Ext_{P_{n}(\delta,G)}^{q}(\1 ,\1) \rightarrow \Ext_{P_{n-1}(\delta,G)}^{q}(\1 ,\1),\]
which are isomorphisms for $2q\leqslant n-1$. This follows from Gan's result \cite[Corollary 6]{Gan} for the homological
stability of wreath products $G\wr \Sigma_n$ (see also
\cite[Proposition 1.6]{HW}). We note that homological stability of
wreath products implies cohomological stability in the same range by the universal
coefficient theorem and the five lemma. The coloured partition algebras join a growing list of algebras which have been shown to exhibit (co)homological stability: the Temperley--Lieb algebras \cite{BH1,Sroka}; the Brauer algebras \cite{BHP}; Iwahori-Hecke
algebras of Type $A$ and Type $B$ \cites{Hepworth, Moselle}; the
partition algebras \cite{BHP2, Boyde2}.

\subsection*{Acknowledgements}
The second author would like to thank Andrew Fisher and Sarah Whitehouse for helpful conversations in work related to this project. He would also like to thank Rachael Boyd and Richard Hepworth for interesting conversations at the 2024 British Topology Meeting in Aberdeen and Guy Boyde for helpful conversations on related work.

The authors would like to thank the Isaac Newton Institute for Mathematical Sciences, Cambridge, for support and hospitality during the programme ``Equivariant homotopy theory in context'', where some of the writing of this paper was undertaken. This work was supported by EPSRC grant EP/Z000580/1.

We would like to thank the anonymous referees for their very detailed and helpful reports, which have improved the paper.

\subsection{Conventions}
Throughout the paper, $k$ will denote a unital, commutative ring. We
will use $\ul{n}$ to denote the set $\{1, \dotsc, n\}$. Throughout the
paper, $G$ will denote a group.

\section{Coloured partition algebras}
\label{CPA-sec}
In this section we recall the definitions of the partition algebras and the $G$-coloured partition algebras. Specifically, we recall these algebras as the endomorphism rings of a certain objects in a $k$-linear category. In this way we recover the partition algebras of Jones and Martin \cite{Jones,Martin-partition} and $G$-coloured partition algebras of Bloss \cite{Bloss}.

\subsection{Partitions and cospans}

In this subsection, we give a description of the \emph{partition
category}. This was introduced by Deligne \cite{Deligne} and has been
further studied in the papers \cite{CO,LS1,Comes}. The
construction we present here is equivalent to those that exist in the
literature but is framed in the language of \emph{cospans}. Throughout this section we will work in the category $\mathbf{Fin}$ of
finite sets and set maps.

Recall that a \emph{partition} of a finite set $A$ is a collection of
non-empty subsets of $A$ such that each element of $A$ lies in
precisely one of the subsets. We call this collection of subsets the
\emph{set of components} of the partition. Furthermore, recall that a
\emph{cospan} in the category $\mathbf{Fin}$ is a diagram of the form
$A\rightarrow U \leftarrow B$. The morphisms $A\rightarrow U$ and
$B\rightarrow U$ will be referred to as the \emph{left} and
\emph{right legs} of the cospan. For fixed $A$ and $B$, two cospans $A\rightarrow U
\leftarrow B$ and $A\rightarrow U^{\prime} \leftarrow B$ are said to
be \emph{equivalent} if there is an isomorphism $U\rightarrow
U^{\prime}$ compatible with the legs of the two cospans.

\begin{defn}
Let $A$ and $B$ be finite sets.  A \emph{partition diagram} from $A$
to $B$ is a partition of $A\sqcup B$. Equivalently, it is an equivalence
class of jointly surjective cospans $A\rightarrow U\leftarrow B$ (that is, an equivalence class of cospans such that the induced map from $A\sqcup B$ to $U$ is a surjection).
\end{defn}

\begin{rem}
The equivalence of the two definitions can be seen as follows. Given a partition
of $A\sqcup B$, we can form a cospan $A\rightarrow U\leftarrow B$ of
the correct form where the legs consist of the maps sending the
elements of $A$ and $B$ to the components which contain them. Conversely,
given an equivalence class of cospans $A\rightarrow U\leftarrow B$,
for each $u\in U$, we take the disjoint union of the preimage of $u$
in $A$ and in $B$. The collection of these sets is the necessary
partition. We can therefore refer to $U$ as the set of components.
\end{rem}

\subsection{Conventions for partition diagrams}

Consider a partition diagram $A\rightarrow U\leftarrow B$. It is natural to draw partition diagrams by drawing graphs (see \cite{Jones,Martin-partition}). We take a
vertex for every element of $A\sqcup B$. An element of $U$ is a subset of $A\rightarrow U\leftarrow B$ and for each element of $U$ we connect the vertices contained in that set. In other words, two vertices lie in the same
component if and only if they lie in the same part of the
partition. We see that there are many possible
representatives of a partition. However, we take all such representatives to be equivalent; the structure of the graph beyond its
components is irrelevant.

In all the diagrams we draw, we will adopt the convention that the
vertices labelled by elements of $A$ are in a single column on the left and the vertices labelled by elements of
$B$ are in a single column on the right, and we will use language in
keeping with this geometric convention as outlined in the following definition.

\begin{defn}
\label{terminology-defn}
Let $A$ and $B$ be finite sets and consider a partition diagram from $A$ to $B$.
  \begin{enumerate}
  \item The elements of $A$ and $B$ will be referred to as \emph{vertices}. 
  \item A pair of elements in the same part will be referred to as an
    \emph{edge};
  \item An edge between $a\in A$ and $b\in B$ will be referred to as
    \emph{propagating}.
    \item A component containing a propagating edge will be called a \emph{propagating component}.
  \item A singleton component will be referred to as an \emph{isolated
  vertex}.
  \item A partition diagram from $A$ to $A$ consisting entirely of
    propagating components of size two will be called a
    \emph{permutation diagram}.
  \end{enumerate}
\end{defn}

\subsection{Partition category}
Partition diagrams form a category. 

\begin{defn}
The partition category $\mathcal{P}$ has finite sets as objects. The homset $\mathcal{P}(A,B)$ consists of all partition diagrams $A\rightarrow U\leftarrow B$. Composition is defined by the
usual recipe for composing (co)spans \cite[Section 2.6]{Benabou}, which ensures
associativity. The composite of $d_1=(A\rightarrow U\leftarrow B)$ and
$d_2=(B\rightarrow V\leftarrow C)$ is $d_1d_2=(A\rightarrow W\leftarrow C)$, where
$W$ is the union of the images of $A$ and $C$ in the pushout
$U\cup_BV$. The identity on $A$ is the cospan $A\rightarrow
A\leftarrow A$ where both maps are identities. When regarded as a partition, it is the partition on
$A\sqcup A$ with all components of size two, each containing matching
elements, one from each summand.
\end{defn}

When we compose in this way, we say that the \emph{number of internal
components} is the size of the subset of $U\cup_BV$ not in the image
of $A$ or $C$. In terms of partitions, this amounts to the following: given
partitions of $A\sqcup B$ and $B\sqcup C$, we form the finest
partition of $A\sqcup B\sqcup C$ so that each part of the two given
partitions is contained in some part, and then we restrict to $A\sqcup
C$. The number of internal components is then the number of components
omitted by this restriction: those contained entirely within $B$.

We can form a linear category out of partition diagrams following \cite[Section 2]{Comes}. In the following category, the objects are the same as $\mathcal{P}$, the homsets are the $k$-linear span of the homsets in $\mathcal{P}$ but we have a different composition rule, which allows us to keep track of the number of internal components formed in composition.

\begin{defn}
\label{lin-part-cat-defn}
Let $\delta\in k$. The \emph{linear partition category} $\cP(\delta)$
is the $k$-linear category with objects finite sets and
homsets $\cP(\delta)(A,B)$ given by the free $k$-module with basis the
set of partition diagrams from $A$ to $B$.

Composition is extended linearly from the following recipe: the
composite of diagrams $d_1$ and $d_2$ is $\delta^nd_1d_2$: the scalar
multiple of the composite diagram $d_1d_2$ (the composite as formed in $\mathcal{P}$) by $\delta$ raised to the
power of the number of internal components in composing $d_1$ and
$d_2$.
\end{defn}

For definiteness, when we are working with partition diagrams from
$\ul{m}$ to $\ul{n}$ we will write $1,\ldots,m$ for elements of the
left-hand summand, and $\ol{1},\ldots,\ol{n}$ for elements of the
right-hand summand.

\begin{defn}
\label{part-alg-defn}
The \emph{partition algebra} $P_n(\delta)$ is the endomorphism algebra $\End_{P(\delta)}(\ul{n})$.
\end{defn}

\subsection{Coloured partitions}

Partition algebras where the edges on a partition diagram are \emph{coloured} by elements of a group were originally defined by Bloss \cite{Bloss}. Categories of coloured partitions  have
been studied in \cite{LS2}. In this subsection we recall the generalized versions of Definitions \ref{lin-part-cat-defn} and \ref{part-alg-defn} which contain the additional data of a $G$-colouring on partitions. We continue to use our terminology from previous
subsections.

\begin{defn}
Let $G$ be a group. A \emph{$G$-colouring} of a partition diagram
consists of a function $\gamma$, defined on pairs of elements in the
same part, and taking values in $G$, such that:
\begin{itemize}
\item $\gamma(x,x)=1$ for any $x\in A\sqcup B$;
\item $\gamma(x,y)=\gamma(y,x)^{-1}$ for any $x,y$ in the same part;
\item $\gamma(x,z)=\gamma(x,y)\gamma(y,z)$ for any $x,y,z$ in the same
  part.
\end{itemize}
If we wish to specify that we are considering the $G$-colouring on a specific diagram $d$, we will use the notation $\gamma_d$.
\end{defn}

\begin{rem}
  A partition on a set $A\sqcup B$ determines an \emph{indiscrete
  groupoid}: objects are $A\sqcup B$, and there is a unique morphism
  from $x$ to $y$ if and only if $x$ and $y$ are in the same part (and
  no morphisms otherwise). The definition of a $G$-colouring amounts exactly to a map of
  groupoids from the indiscrete groupoid on the partition to $G$.
\end{rem}

\begin{defn}
  Given any partition diagram, there is the \emph{trivial}
  $G$-colouring, where $\gamma(x,y)=1$ for any pair of elements $x$
  and $y$ in the same part of the partition.
\end{defn}

\begin{rem}
  A component of size $n$ of a partition has $|G|^{n-1}$ different
  $G$-colourings. Indeed, any colouring can be determined by fixing an
  element $x$, defining $\gamma(x,y)$ for all $y\neq x$, and extending
  the definition to all elements using $\gamma(x,x)=1$ and
  $\gamma(y_1,y_2)=\gamma(x,y_1)^{-1}\gamma(x,y_2)$.

  A more canonical construction is to choose $\alpha_x\in G$ for each
  element $x$, and to define $\gamma(x,y)=\alpha_x^{-1}\alpha_y$. This
  construction is unchanged by taking $\alpha'_x=g\alpha_x$, so the
  set of $G$-colourings of a component of size $n$ is naturally
  isomorphic to the quotient of $G^n$ by the free action of $G$ by
  left multiplication.
\end{rem}

Partition diagrams equipped with $G$-colourings form a partial
category: if a $G$-colouring of a partition of $A\sqcup B$ and a
$G$-colouring of $B\sqcup C$ can be extended to a $G$-colouring of a
partition of $A\sqcup B\sqcup C$ then it can be extended uniquely. 

As for partitions, we can form a linear category of coloured partitions. Intuitively speaking, the associativity of the product follows from the associativity of composition for the underlying partitions and the associativity of multiplication in a group.

\begin{defn}
The \emph{linear $G$-partition category} $\cP(\delta,G)$ is the
linear category whose objects are finite sets and whose homsets are
free $k$-modules with basis the set of $G$-coloured partition
diagrams. The composition of $G$-coloured diagrams $d_1$ with $d_2$ is
$\delta^nd_1d_2$ (where $n$ is the number of internal components) if
the $G$-colourings extend to a $G$-colouring of $d_1d_2$, and zero
otherwise.
\end{defn}

\begin{defn}
The \emph{$G$-partition algebra}, $P_n(\delta,G)$, is the endomorphism algebra $\End_{\cP(\delta,G)}(\ul{n})$.
\end{defn}

\subsection{Functoriality}

The coloured partition algebras are functorial in $G$. Given a group homomorphism $G\rightarrow H$ there is an induced map $P_n(\delta , G)\rightarrow P_n(\delta , H)$ determined on basis diagrams by applying the group homomorphism to labels in the colouring. This is compatible with composition and reversal of arrows in the coloured partition algebras since group homomorphisms are compatible with group multiplication and taking inverses.

In particular, since the trivial group is a retract of any group $G$, the usual uncoloured partition algebra $P_n(\delta)$ is always a retract of a coloured partition algebra $P_n(\delta , G)$. Furthermore, this retraction fits into a commutative square with the retraction of the symmetric group algebra from the wreath product group algebra.

\subsection{Conventions for coloured partition diagrams}

It is also natural to draw $G$-coloured partitions using graphs. Our graphs will be drawn as follows. We will have two identical columns of vertices drawn in parallel. The vertices in the left-hand column will be labelled $1$ up to $n$ from top to bottom. The vertices in the right-hand column will be labelled $\ol{1}$ up to $\ol{n}$ from top to bottom. In order to
capture the $G$-colouring they must be directed graphs and edges
must be labelled by elements of $G$. To find $\gamma(x,y)$, one
multiplies the edge labels of a path from $x$ to $y$; traversing an
edge labelled by $g$ in the reverse direction amounts to traversing an
edge labelled by $g^{-1}$. We implicitly assume that, if there are
multiple paths, then they will have the same product. An equivalent
condition, which is easier to check, is that any loops have product
$1\in G$.

Regarded in such language, a composite of partitions is zero if it
would form a loop around which the product is non-trivial.

In order to keep our diagrams simple, if edges are not drawn in a
component, we mean that that component has the trivial
$G$-colouring.

We note that there is an inclusion map $P_n(\delta,G)\rightarrow
P_{n+1}(\delta,G)$. This sends a $G$-coloured partition, $d$, of the
set $\ul{n}$ to the $G$-coloured partition of the set $\ul{n+1}$
consisting of the components of $d$ and the component $\{n+1,
\ol{n+1}\}$. The $G$-colouring, $\gamma$, of $d$ is extended to this
new partition by defining $\gamma(n+1,\ol{n+1})=1\in G$.

\subsection{Defining the augmentation}
\label{cohom-defn-subsec}
Recall that a $k$-algebra $A$ is said to be \emph{augmented} if it
comes equipped with a $k$-algebra map $\varepsilon\colon A \rightarrow
k$.

As for partition diagrams (see \cite[Proposition 2]{Martin} for instance), the composite of two coloured partition
$n$-diagrams, having $i$ and $j$ propagating components respectively,
has at most $\min(i,j)$ propagating components. We can therefore make
the following definition. Let $I_{n-1}$ denote the two-sided ideal in
$P_n(\delta , G)$ spanned $k$-linearly by non-permutation diagrams.

\begin{lem}
\label{partition-ideal-iso-lem}
There is an isomorphism of $k$-algebras $P_n(\delta,G)/I_{n-1} \cong
k[G\wr \Sigma_n]$.
\end{lem}
\begin{proof}
The two-sided ideal $I_{n-1}$ has a $k$-module basis of
$G$-coloured partition $n$-diagrams with fewer than $n$
propagating components. The quotient therefore has a $k$-module basis
consisting of the permutation diagrams (that is, the $G$-coloured
partition $n$-diagrams with precisely $n$-propagating edges) with the product induced from $P_n(\delta,G)$. This product is precisely the composition in the wreath product written in terms of labelled permutation diagrams (see \cite[Subsection 4.1]{Bloss} for instance).
\end{proof}

\begin{defn}
We equip $P_n(\delta, G)$ with an augmentation $\varepsilon\colon
P_n(\delta,G) \rightarrow k$ that sends the permutation diagrams to $1
\in k$ and all non-permutation diagrams to $0\in k$.  The
\emph{trivial module}, $\mathbbm{1}$, consists of a single copy of the
ground ring $k$, where $P_n(\delta , G)$ acts on $k$ via the
augmentation.
\end{defn}

\section{Proving Theorem \ref{thm-A}}
\label{idem-cover-sec}

We begin by recalling the definition of $k$-free idempotent left cover and its associated Mayer-Vietoris complex from \cite{Boyde2}. Having recalled the definition, we prove that the
two-sided ideal $I_{n-1}$ introduced in the previous section has a
$k$-free idempotent left cover. This is the key technical work
required to deduce our main result.

\subsection{Idempotent covers and the Mayer-Vietoris complex}

\begin{defn}
Let $A$ be a $k$-algebra. Let $I$ be a two-sided ideal of $A$. Let
$w\geqslant h \geqslant 1$. An \emph{idempotent left cover of $I$ of
height $h$ and width $w$} is a collection of left ideals $J_1,\dotsc ,
J_w$ in $A$ such that
\begin{itemize}
    \item $J_1+\cdots +J_w=I$;
    \item for $S\subset \ul{w}$ with $\lvert S\rvert \leqslant h$, the
      intersection
    \[\bigcap_{i\in S} J_i\]
    is either zero or is a principal left ideal generated by an
    idempotent.
\end{itemize}
If $I$ is free as a $k$-module, then an idempotent left cover is said
to be \emph{$k$-free} if there is a choice of $k$-basis for $I$ such
that each $J_i$ is free on a subset of this basis.
\end{defn}

\begin{defn}
\label{MV-defn}
Let $A$ be a $k$-algebra. Let $I\subset A$ be a two-sided ideal. Let $J_1,\dotsc , J_w$ be an idempotent left cover of $I$. The \emph{Mayer-Vietoris complex associated to the idempotent left cover}, $C_{\star}$, is the chain complex of left $A$-modules defined as follows. We set
\[C_p= \underset{\llv S \rrv = p}{\bigoplus_{S \subset \ul{w}}}\bigcap_{i\in S} J_i\]
for $1\leqslant p \leqslant w$. We set $C_0=A$, $C_{-1}=A/I$ and $C_n=0$ for $n>w$ and $n<-1$.

The differential $C_0\rightarrow C_{-1}$ is the projection map $A\rightarrow A/I$. The differential $C_1\rightarrow C_0$ is the direct sum of the inclusion of the left ideals $J_i\rightarrow A$. For $p\geqslant 2$, the differential $C_p\rightarrow C_{p-1}$ is defined on the summand $\cap_{i\in S} J_i$ by
\[x\mapsto \sum_{j\in S} (-1)^{\#(S,j)} i_{(S,j)}(x)  \]
where $\#(S,j)$ is the number of elements of $S$ that are less than $j$ and $i_{(S,j)}$ is the inclusion
\[\bigcap_{i\in S} J_i \rightarrow \bigcap_{i\in S\setminus \llb j\rrb} J_i.\]
\end{defn}

We will make use of the following result, which can be proved by showing that if a $k$-free idempotent cover has height $h$, then the Mayer-Vietoris complex is a partial projective resolution of length $h$. This can be found in \cite[Theorem B]{Fisher-Graves}, with the homological statement originally due to Boyde \cite[Theorem 1.7]{Boyde2}. 

\begin{prop}
\label{Boyde-FG-prop}
Let $A$ be an augmented $k$-algebra with trivial module $\mathbbm{1}$. Let $I$ be a two-sided ideal of $A$ which is free as a $k$-module and which acts as multiplication by $0\in k$ on $\mathbbm{1}$. Suppose that there exists a $k$-free idempotent left cover of $I$ of height $h$ and width $w$. There are natural isomorphisms of $k$-modules
\[\Tor_q^A(\mathbbm{1},\mathbbm{1}) \cong \Tor_q^{A/I}(\mathbbm{1},\mathbbm{1})
\quad
\text{and} \quad
\Ext_A^q(\mathbbm{1},\mathbbm{1}) \cong \Ext_{A/I}^q(\mathbbm{1},\mathbbm{1})\]
for $q\leqslant h$. Furthermore, if $h=w$, then the isomorphisms holds for all $q$.
\end{prop}

\subsection{A free idempotent left cover}
We will construct a $k$-free
idempotent left cover of the two-sided ideal $I_{n-1}$ in
$P_n(\delta,G)$ and use Proposition \ref{Boyde-FG-prop} to deduce stable
isomorphisms on (co)homology. We begin by defining our families of left ideals.

\begin{defn}
For $1\leqslant i \leqslant n$, let $K_i$ denote the left ideal of
$P_n(\delta ,G)$ spanned $k$-linearly by $G$-coloured partition
$n$-diagrams such that $\ol{i}$ is an isolated vertex.
\end{defn}

\begin{defn}
\label{ideals-L-defn}
Let $1\leqslant i <j \leqslant n$. For each $g\in G$, let $L_{i,j,g}$
denote the left ideal of $P_n(\delta ,G)$ spanned $k$-linearly by
$G$-coloured partition $n$-diagrams with $\ol{i}$ and $\ol{j}$ in the
same component and $\gamma(\ol{i},\ol{j})=g$.
\end{defn}

\begin{lem}
The collection of ideals $K_i$ for $1\leqslant i \leqslant n$ and
$L_{i,j,g}$ for $1\leqslant i < j \leqslant n$ and $g\in G$ cover the
two-sided ideal $I_{n-1}$.
\end{lem}
\begin{proof}
Any basis diagram of $K_i$ for $1\leqslant i \leqslant n$ has an
isolated vertex at $\ol{i}$ and so can have at most $n-1$ propagating
components. Similarly, any basis diagram in $L_{i,j,g}$ for
$1\leqslant i < j \leqslant n$ has a connected component with (at
least) two vertices in the right-hand column and so can have at most
$n-1$ propagating components. Therefore, each $K_i$ and each
$L_{i,j,g}$ is contained in $I_{n-1}$.

Consider a basis diagram $d$ in $I_{n-1}$. If two vertices $\ol{i}$
and $\ol{j}$ are in the same connected component, then $d$ lies in
some $L_{i,j,g}$. Otherwise, $d$ must contain an isolated vertex in
the right-hand column (if not, each vertex in the right-hand column
would be connected to a distinct vertex in the left-hand column, which
contradicts the fact that $d$ has fewer than $n$ propagating
components) and so $d$ lies in some $K_i$. Therefore every basis
element of $I_{n-1}$ is contained in (at least) one of the ideals
$K_i$ and $L_{i,j,g}$.
\end{proof}

\begin{defn}
Let $\ul{n}_{<}^2$ denote the set of indices $(i,j)$ with $1\leqslant
i< j \leqslant n$.
\end{defn}

\begin{rem}
\label{intersection-zero-rem}
We will be considering intersections of these ideals. We note the following.
\begin{itemize}
\item For $g\neq h$, the intersection $L_{i,j,g}\cap L_{i,j,h}$ is zero since the two ideals dictate different colourings on a non-propagating edge from vertex $\ol{i}$ to vertex $\ol{j}$.
\item Suppose we have an intersection of the form $L_{i,j,g}\cap L_{j,k,h}$. If we now take the intersection with $L_{i,k,gh}$ then nothing changes since the edge from $\ol{i}$ to $\ol{k}$ with colouring $gh$ was already determined by the definitions of $L_{i,j,g}$ and $L_{j,k,h}$.
\item On the other hand, if we take the intersection with any ideal of the form $L_{i,k,g^{\prime}}$ with $g^{\prime}\neq gh$ then the intersection would be zero since the  ideals $L_{i,j,g}$ and $L_{j,k,h}$ determine an edge from $\ol{i}$ to $\ol{k}$ with label $gh$ but the ideal $L_{i,k,g^{\prime}}$ dictates an edge from $\ol{i}$ to $\ol{k}$ with label $g^{\prime}$. In other words the intersection $L_{i,j,g}\cap L_{j,k,h}$ and the ideal $L_{i,k,g^{\prime}}$ give contradictory colourings. Indeed, the first bullet point in this list is the simplest example of this.
\end{itemize} 
This will play a role in Lemma \ref{intersection-zero-lem} below.
\end{rem}

This observation leads us to the following definition.

\begin{defn}
\label{contradictory-colouring-defn}
Let $T\subset \ul{n}_{<}^2\times G$. We say that $T$ \emph{determines contradictory colourings} if there exist subsets $T_1,\,T_2\subseteq T$ such that 
\begin{enumerate}
\item a basis element of
\[\bigcap_{((i,j),g)\in T_1} L_{i,j,g}\]
must have an edge from vertex $\ol{v}_1$ to vertex $\ol{v}_2$ with colouring $g_1$ and
\item a basis element of
\[\bigcap_{((i,j),g)\in T_2} L_{i,j,g}\]
must have an edge from vertex $\ol{v}_1$ to vertex $\ol{v}_2$ with colouring $g_2\neq g_1$.
\end{enumerate}
\end{defn}

We began by noting two criteria for intersections of our ideals to be zero.

\begin{lem}
\label{intersection-zero-lem}
Let $S\subset \ul{n}$ and let $T\subset \ul{n}_{<}^2\times G$. The intersection
\[J=\bigcap_{i\in S} K_i \cap \bigcap_{((i,j),g)\in T} L_{i,j,g}\]
is zero if and only if one of the following conditions holds:
\begin{enumerate}
    \item there exists $((i,j),g)\in T$ such that $i\in S$ or $j\in S$;
    \item the set $T$ determines contradictory colourings (in the sense of Definition \ref{contradictory-colouring-defn}). 
\end{enumerate}
\end{lem}
\begin{proof}
Suppose there exists $((i,j),g) \in T$ with $i\in S$ or $j\in
S$. Suppose $J\neq 0$. Since $i\in S$ (or $j\in S$), a diagram in $J$
has an isolated vertex at $\ol{i}$ (or $\ol{j}$). However, since
$((i,j),g) \in T$, $\ol{i}$ must be in the same component as
$\ol{j}$. This is a contradiction and so $J=0$. It follows from the argument in Remark
\ref{intersection-zero-rem} that if $T$ determines contradictory colourings then the intersection 
\[\bigcap_{((i,j),g)\in T} L_{i,j,g}\]
is zero and, therefore, $J=0$.

Conversely, suppose neither of the conditions holds. In this case,
the $G$-coloured partition $n$-diagram with $\ol{i}$ and $\ol{j}$ in
the same component and $\gamma(\ol{i},\ol{j})=g$ for each index
$((i,j),g)\in T$ and isolated vertices elsewhere is an element of $J$,
so $J\neq 0$.
\end{proof}

\begin{lem}
\label{mu-lem}
Let $S\subset \ul{n}$ and let $T\subset \ul{n}_{<}^2\times G$. Let
\[J=\bigcap_{i\in S} K_i \cap \bigcap_{((i,j),g)\in T} L_{i,j,g}.\]
Let $a,b\in \ul{n}\setminus S$ with $a\neq b$.

Let $\mu$ be the partition $n$-diagram whose connected components are
$\{\ol{a}\}$, $\{a, b, \ol{b}\}$ and $\{i, \ol{i}\}$ for $i\neq a, b$,
equipped with the trivial $G$-colouring.
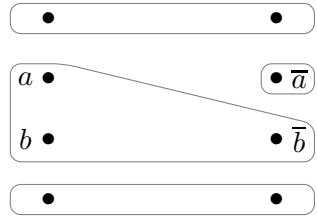
\begin{displaymath}
\begin{tikzpicture}
\node at (0,0.0) {$\bullet$};
\node at (0,0.8) {$\bullet$};
\node at (0,1.6) {$\bullet$};
\node at (0,2.4) {$\bullet$};
\node at (3,0.0) {$\bullet$};
\node at (3,0.8) {$\bullet$};
\node at (3,1.6) {$\bullet$};
\node at (3,2.4) {$\bullet$};
\node at (-0.3,1.6) {$a$};
\node at (3.3,1.6) {$\ol{a}$};
\node at (-0.3,0.8) {$b$};
\node at (3.3,0.8) {$\ol{b}$};
\draw[black!50, rounded corners] (-0.5,2.6) rectangle (3.5,2.2);
\draw[black!50, rounded corners] (-0.5,1.8) -- (0.2,1.8) -- (3.5,1.0) -- (3.5,0.5) -- (-0.5,0.5) -- cycle;
\draw[black!50, rounded corners] (2.8,1.4) rectangle (3.5,1.8);
\draw[black!50, rounded corners] (-0.5,0.2) rectangle (3.5,-0.2);
\end{tikzpicture}
\end{displaymath}
Then either $K_a \cap J = 0$ or $J\cdot \mu \subset K_a \cap J$.
\end{lem}
\begin{proof}
Suppose $K_a \cap J \neq 0$ and let $d\in J$ be a partition. For each $i\in S$,
$\ol{i}$ is an isolated vertex in $d$, since $d\in J$. Since $i \neq
a$ and $i\neq b$, the vertex $\ol{i}$ is also isolated in the
composite $d\mu$, so $d\mu \in K_i$ for each $i\in S$. If
$((i,j),g)\in T$, then $\ol{i}$ and $\ol{j}$ are in the same component
and $\gamma(\ol{i},\ol{j})=g$ in $d$. Since $K_a \cap J \neq 0$,
neither $i$ nor $j$ can be equal to $a$ by Lemma
\ref{intersection-zero-lem}. Therefore $i$ and $\ol{i}$ lie in the same connected component of $\mu$. Similarly, $j$ and $\ol{j}$ lie in the same connected component of $\mu$. It follows that $\ol{i}$ and $\ol{j}$ are in the same connected component in the composite $d\mu$.
Since
$\gamma_{\mu}(i,\ol{i})=\gamma_{\mu}(j,\ol{j})=1$, we have
$\gamma_{d\mu}(\ol{i},\ol{j})=g$. Therefore, $d\mu \in J$.

Finally, since $\ol{a}$ is an isolated vertex in $\mu$ it must also be
an isolated vertex in $d\mu$, so $d\mu \in K_a$. Therefore $d \mu \in
K_a \cap J$ as required.
\end{proof}

\begin{lem}
\label{retract-lem-1}
Let $S\subset \ul{n}$ and let $T\subset \ul{n}_{<}^2\times G$. Let
\[J=\bigcap_{i\in S} K_i \cap \bigcap_{((i,j),g)\in T} L_{i,j,g}.\]
Let $a,b\in \ul{n}\setminus S$, with $a\neq b$. If $K_a \cap J$ is
non-zero, then right multiplication by the element $\mu$ (of Lemma
\ref{mu-lem}) gives a retraction of the inclusion $K_a \cap J
\rightarrow J$.
\end{lem}
\begin{proof}
For each partition $d\in K_a \cap J$, we must show that $d\mu=d$.

By construction, every vertex in the left-hand column of $\mu$ is
connected to a vertex in the right-hand column and so the composite
$d\mu$ can have no factors of $\delta$. Furthermore, $\ol{a}$ must be
an isolated vertex in $d$ (and, in particular, is not in the same
component as $\ol{b}$) so the composite cannot be zero due to a
mismatch of labels between a non-propagating edge in the right-hand
column of $d$ and a non-propagating edge in the left-hand column of
$\mu$.

Therefore, $\mu$ acts on $d$ by removing
$\ol{a}$ from its connected component and merging the remainder with
the connected component containing $\ol{b}$, whilst preserving all
other connected components. However, since $\ol{a}$ is already an
isolated vertex in $d$, we have $d\mu = d$ as required.
\end{proof}

\begin{lem}
\label{nu-lem}
Let $S\subset \ul{n}$ and let $T\subset \ul{n}_{<}^2\times G$. Let
\[J=\bigcap_{i\in S} K_i \cap \bigcap_{((i,j),g)\in T} L_{i,j,g}.\]
Let $((a,b),h) \in (\ul{n}_{<}^2 \times G) \setminus T$.

We write $\nu$ for the partition $n$-diagram whose connected
components are $\{a,b,\ol{a},\ol{b}\}$ and $\{i, \ol{i}\}$ for $i\neq
a, b$, with the $G$-colouring defined by
\begin{displaymath}
  \gamma_{\nu}(x,y)=
  \begin{cases}
    h,     & \text{if $x\in\{a,b\}$ and $y\in\{\ol{a},\ol{b}\}$;}\\
    h^{-1}, & \text{if $x\in\{\ol{a},\ol{b}\}$ and $y\in\{a,b\}$;}\\
    1,     & \text{otherwise}.
  \end{cases}
\end{displaymath}
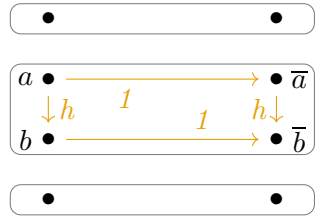
\begin{displaymath}
\begin{tikzpicture}
\node at (0,0.0) {$\bullet$};
\node at (0,0.8) (b) {$\bullet$};
\node at (0,1.6) (a) {$\bullet$};
\node at (0,2.4) {$\bullet$};
\node at (3,0.0) {$\bullet$};
\node at (3,0.8) (bbar) {$\bullet$};
\node at (3,1.6) (abar) {$\bullet$};
\node at (3,2.4) {$\bullet$};
\node at (-0.3,1.6) {$a$};
\node at (3.3,1.6) {$\ol{a}$};
\node at (-0.3,0.8) {$b$};
\node at (3.3,0.8) {$\ol{b}$};
\draw[black!50, rounded corners] (-0.5,2.6) rectangle (3.5,2.2);
\draw[black!50, rounded corners] (-0.5,1.8) rectangle (3.5,0.6);
\draw[black!50, rounded corners] (-0.5,0.2) rectangle (3.5,-0.2);
\draw[bw3, ->] (a) -- (b) node[midway, right] {h};
\draw[bw3, ->] (a) -- (abar) node[pos=0.3, below] {1};
\draw[bw3, ->] (abar) -- (bbar) node[midway, left] {h};
\draw[bw3, ->] (b) -- (bbar) node[pos=0.7, above] {1};
\end{tikzpicture}
\end{displaymath}
Then either $L_{a,b,h} \cap J = 0$ or $J\cdot \nu \subset L_{a,b,h} \cap J$.
\end{lem}
\begin{proof}
Suppose $L_{a,b,h} \cap J \neq 0$ and let $d\in J$. For $i\in S$,
$\ol{i}$ is an isolated vertex in $d$. Since $L_{a,b,h} \cap J \neq
0$, we have $a\neq i$ and $b\neq i$ by Lemma
\ref{intersection-zero-lem}. Therefore $\{i ,\ol{i}\}$ is a
connected component of $\nu$ and the composite $d\nu$ also has an
isolated vertex at $\ol{i}$. Therefore $d\nu \in K_i$ for $i\in
S$. 

For $((i,j),g)\in T$, $\ol{i}$ and $\ol{j}$ are in the same
component in $d$ and $\gamma_{d}(\ol{i},\ol{j})=g$. Since $L_{a,b,h} \cap
J \neq 0$ and $((a,b),h) \in (\ul{n}_{<}^2 \times G) \setminus T$, it
follows that $(i,j)$ is distinct from $(a,b)$. Therefore $\ol{i}$ and
$\ol{j}$ are in the same component in $d\nu$, and
$\gamma_{d\nu}(\ol{i},\ol{j})=g$ (this can be deduced as
$\gamma_{d\nu}(\ol{i},\ol{j}) =
\gamma_{\nu}(\ol{i},i)\gamma_{d}(\ol{i},\ol{j})\gamma_{\nu}(j,\ol{j})$). Hence $d\nu \in
L_{i,j,g}$ for each $((i,j),g)\in T$.

Finally, $\ol{a}$ and $\ol{b}$ are in the same component in $\nu$ and
$\gamma_{\nu}(\ol{a},\ol{b})=h$, so $d\nu \in L_{a,b,h}$ as
required.
\end{proof}

\begin{lem}
\label{retract-lem-2}
Let $S\subset \ul{n}$ and let $T\subset \ul{n}_{<}^2 \times G$. Let
\[J=\bigcap_{i\in S} K_i \cap \bigcap_{((i,j),g)\in T} L_{i,j,g}.\]
Let $((a,b),h) \in (\ul{n}_{<}^2 \times G) \setminus T$. If $L_{a,b,h}
\cap J$ is non-zero, then right multiplication by the element $\nu$
(of Lemma \ref{nu-lem}) gives a retraction of the inclusion
$L_{a,b,h} \cap J \rightarrow J$.
\end{lem}
\begin{proof}
For each partition $d\in L_{a,b,h} \cap J$, we must show that $d\nu=d$.

By construction, every vertex in the left-hand column of $\nu$ is
connected to a vertex in the right-hand column and so the composite
$d\nu$ can have no factors of $\delta$. The diagram $\nu$ has one
non-propagating edge in the left-hand column and this corresponds to a
non-propagating edge in the right-hand column of $d$ with matching
label and so the composite is non-zero.

Therefore, right multiplication by $\nu$
acts on $d$ by merging the connected components containing $\ol{a}$
and $\ol{b}$, whilst preserving all other connected components. Since
$\ol{a}$ and $\ol{b}$ are already connected in $d$, we have $d\nu=d$
as required.
\end{proof}

\begin{lem}
\label{height-lem}
If $S\neq \ul{n}$, then the left ideal
\[J=\bigcap_{i\in S} K_i \cap \bigcap_{((i,j),g)\in T} L_{i,j,g}\]
is either zero or principal and generated by an idempotent.
\end{lem}
\begin{proof}
Suppose $J\neq 0$. Lemma \ref{intersection-zero-lem} tells us that the sets $S$ and $\bigcup_{(i,j)\in T} \lbrace i,j\rbrace$ are disjoint. It also tells us that there does not exist a combination of indices in $T$ preventing a valid colouring.

We
inductively apply Lemma \ref{retract-lem-1} to get a retraction
$P_n(\delta,G)\rightarrow \bigcap_{i\in S} K_i$ (which is valid
because of the condition $S\neq \ul{n}$) and then inductively apply
Lemma \ref{retract-lem-2} to obtain a retraction $ \bigcap_{i\in S}
K_i \rightarrow J$. These retractions are given by right multiplication by $G$-coloured partition diagrams of the form $\mu$ and $\nu$. It now follows from \cite[Lemma 2.5]{Boyde2} that $J$ is principal and generated by an idempotent.
\end{proof}

We can now prove most of Theorem \ref{thm-A}. In the next subsection, we will extend the range by one.

\begin{thm}
\label{iso-thm1}
There exist natural isomorphisms of $k$-modules
\[ \Tor_{q}^{P_n(\delta,G)}(\mathbbm{1},\mathbbm{1}) \cong \Tor_q^{k[G\wr  \Sigma_n]}(\mathbbm{1},\mathbbm{1}) \quad \text{and} \quad \Ext_{P_n(\delta,G)}^q(\mathbbm{1},\mathbbm{1}) \cong \Ext_{k[G\wr \Sigma_n]}^q(\mathbbm{1},\mathbbm{1})\]
for $q\leqslant n-1$. Furthermore, if $\delta$ is invertible, these isomorphisms hold for all $q$.
\end{thm}
\begin{proof}
Let $\delta$ be any element of $k$. Recall that $I_{n-1}$ is the two-sided ideal of $P_n(\delta,G)$
spanned $k$-linearly by the non-permutation diagrams. By Lemma
\ref{partition-ideal-iso-lem}, there is an isomorphism of $k$-algebras
$P_n(\delta,G)/I_{n-1} \cong k[G\wr \Sigma_n]$.

Lemma \ref{height-lem} tells us that the left ideals $K_i$ and
$L_{i,j,g}$ form a $k$-free idempotent left cover of $I_{n-1}$ of
height $n-1$. The isomorphisms now follow from
Proposition \ref{Boyde-FG-prop}.

By Lemmas \ref{intersection-zero-lem} to \ref{height-lem}, the only non-zero intersection of ideals in our idempotent cover that is not principal and generated by an idempotent is the $n$-fold intersection $K_1\cap \cdots \cap K_n$.

Now suppose that $\delta$ is invertible. It suffices to show that, in this case, the intersection $K_1\cap \cdots \cap K_n$ is principal and generated by idempotent. We will then have a $k$-free idempotent left cover whose height is equal to the width and the result follows from Proposition \ref{Boyde-FG-prop}.

Let $d$ be the coloured partition $n$-diagram such that all $2n$ vertices are isolated. Then for any $y\in K_1\cap \cdots \cap K_n$, we have $yd=\delta^n y$. In particular, this tells us that $\delta^{-n}d$ is idempotent and that right multiplication by $\delta^{-n}d$ gives a retraction $P_n(\delta,G)\rightarrow K_1\cap \cdots \cap K_n$. It follows that $K_1\cap \cdots \cap K_n$ is principal and generated by an idempotent. 
\end{proof}

\subsection{Extending the range}

Throughout this section, we do not assume that $\delta\in k$ is invertible. We begin by describing the terms in our Mayer-Vietoris complex. We show that the (co)homology of the coloured partition algebras over the group algebras $k[G\wr \Sigma_n]$ can be described in terms of the $n$-fold intersection of ideals $K_1\cap \cdots \cap K_n$. Finally, we feed this into a change-of-rings spectral sequence.

Recall the Mayer-Vietoris complex from Definition \ref{MV-defn}. We begin by analysing the Mayer-Vietoris complex, $C_{\star}$, associated to our $k$-free idempotent left cover. Recall that $C_q$ takes the form of a direct sum of $q$-fold intersections of ideals in our $k$-free idempotent left cover. We have shown that the terms $C_0$ up to $C_{n-1}$ in the Mayer-Vietoris form a partial projective resolution of $\1$ by left $k[G\wr\Sigma_n]$-modules and used this to deduce the range in Theorem \ref{thm-A}. We now wish to consider the remaining terms in the Mayer-Vietoris complex.

\begin{lem}
\label{K-intersection-lem}
We have
\begin{itemize}
\item $\1\otimes_{P_n(\delta, G)} C_n=\1\otimes_{P_n(\delta, G)} (K_1\cap \cdots \cap K_n)$ and
\item $\Hom_{P_n(\delta, G)}(C_n,\1)=\Hom_{P_n(\delta, G)}(K_1\cap \cdots \cap K_n,\1)$.
\end{itemize}
Furthermore, $\1\otimes_{P_n(\delta, G)} C_q=0$ and $\Hom_{P_n(\delta, G)}(C_q,\1)=0$ for $q>n$.
\end{lem}
\begin{proof}
We have shown that any $(n-1)$-fold intersection of ideals in our cover is either zero or is principal and generated by an idempotent. We have two types of non-zero $n$-fold intersections: the intersection $K_1\cap \cdots \cap K_n$ and intersections of the form considered in Lemma \ref{height-lem}. The intersections of the form considered in Lemma \ref{height-lem} are generated by idempotents and their basis diagrams act on $\1$ as multiplication by $0\in k$. Therefore \cite[Lemma 2.3]{Boyde2} and \cite[Lemma 3.4]{FG-dTL} tell us that we get zero when we apply the functors $\1 \otimes _{P_n(\delta, G)} -$ and $\Hom_{P_n(\delta, G)}(-,\1)$ to these intersections. We are left with the intersection $K_1\cap \cdots \cap K_n$ and we obtain the first part of the statement. Furthermore, any $(n+1)$-fold intersections (or indeed any intersection of $q>n$ ideals) is either zero or of the form considered in Lemma \ref{height-lem} so $\1\otimes_{P_n(\delta, G)} C_q=0$ and $\Hom_{P_n(\delta, G)}(C_q,\1)=0$ for $q>n$
\end{proof}

We calculate the (co)homology of the coloured partition algebra over the group algebra.

\begin{prop}
\label{hom-over-group-ss-prop}
There exist isomorphisms of $k$-modules
\[\Tor_{q}^{P_n(\delta,G)}(\1,k[G\wr \Sigma_n])\cong \begin{cases}
k & q=0\\
0 & 1\leqslant q\leqslant n-1\\
\Tor_{q-n}^{P_n(\delta,G)}(\1,K_1\cap \cdots \cap K_n) & q\geqslant n.
\end{cases}\]
\end{prop}
\begin{proof}
Let $Q_{\star}$ be a projective resolution of $\1$ by right $P_n(\delta,G)$-modules and consider the double complex $Q_{\star}\otimes_{P_n(\delta,G)} C_{\star}$, with the boundary maps from $Q_{\star}$ in the horizontal direction and the boundary maps from $C_{\star}$ in the vertical direction. Since $C_{\star}$ is acyclic with an augmentation to $k[G\wr\Sigma_n]$, and since $Q_{\star}$ is a projective resolution of $\1$, we see that the vertical-homology-first spectral sequence collapses on the $E^2$-page where it consists of the groups $\Tor_{q}^{P_n(\delta,G)}(\1,k[G\wr \Sigma_n])$ concentrated in row zero.

The horizontal-homology-first spectral sequence collapses on the $E^1$-page with
\[E_{\alpha,\beta}^1\cong \begin{cases}
k & (\alpha,\beta)=(0,0),\\
\Tor_{\alpha}^{P_n(\delta,G)}(\1,K_1\cap \cdots \cap K_n) & (\alpha,\beta)=(\alpha,n),\\
0 & \text{otherwise}.
\end{cases}\] 
When $n=1$, there is a differential on the $E^1$-page which could possibly be non-zero, namely the map $E_{0,1}\rightarrow E_{0,0}$ induced from the differential $\1\otimes_{P_1(\delta,G)} K_1\rightarrow \1\otimes_{P_1(\delta,G)} P_1(\delta,G)$ in the Mayer-Vietoris complex. However, this map sends $\lambda\otimes d \mapsto \lambda\otimes d= \lambda\otimes d\cdot 1=\lambda\cdot d\otimes 1=0$ and so the differential on the $E^1$-page is zero. Therefore, the horizontal-homology-first spectral sequence always collapses on the $E^1$-page and we can read off the result.
\end{proof}

\begin{prop}
\label{cohom-over-group-ss-prop}
There exist isomorphisms of $k$-modules
\[\Ext_{P_n(\delta,G)}^q(k[G\wr \Sigma_n],\1)\cong \begin{cases}
k & q=0\\
0 & 1\leqslant q\leqslant n-1\\
\Ext_{P_n(\delta,G)}^{q-n}(K_1\cap \cdots \cap K_n,\1) & q\geqslant n.
\end{cases}\]
\end{prop}
\begin{proof}
The result is proved in a similar way to Proposition \ref{hom-over-group-ss-prop}. We only state the key differences in the proof. Let $I^{\star}$ be an injective resolution of $\1$ by left $P_n(\delta,G)$-modules and consider the bicomplex $\Hom_{P_n(\delta,G)}(C_{\star},I^{\star})$ with the boundary maps from $C_{\star}$ in the horizontal direction and the boundary maps from $I^{\star}$ in the vertical direction. The horizontal-homology-first spectral sequence collapses on the $E_2$-page, where it consists of the groups $\Ext_{P_n(\delta,G)}^{\star}(k[G\wr\Sigma_n],\1)$ concentrated in column zero. Furthermore, the vertical-homology-first spectral sequence collapses on the $E_1$-page with 
\[E_1^{\alpha,\beta}\cong \begin{cases}
k & (\alpha , \beta) = (0,0)\\
\Ext_{P_n(\delta,G)}^{\beta}(K_1\cap\cdots \cap K_n,\1) & (\alpha,\beta)=(n,\beta)\\
0 & \text{otherwise}
\end{cases}\] 
as required. (Note that in this case, when $n=1$ there is a differential on the $E_1$-page which could possibly have been non-zero, namely the map $E_{1}^{0,0}\rightarrow E_{1}^{1,0}$. This is the restriction map $\Hom_{P_1(\delta,G)}(P_1(\delta,G),\1)\rightarrow \Hom_{P_1(\delta,G)}(K_1,\1)$. However, any $f\colon P_1(\delta,G)\rightarrow\1$ restricts to the zero map on $K_1$: we use linearity to write $f(d)=df(1)$ and note that basis diagrams in $K_1$ act as multiplication by $0\in k$ on $\1$.)
\end{proof}

We will use these propositions to show that the $n^{\mathrm{th}}$ (co)homology group of the coloured partition algebras is isomorphic to the $n^{\mathrm{th}}$ (co)homology group of the wreath product. Before doing so, we require the following lemma.

\begin{lem}
\label{K-decomp-lem}
For $n\geqslant 2$, every basis diagram $d\in K_1\cap \cdots \cap K_n$ can be written as $d^{\prime}d$ where $d^{\prime}\in P_n(\delta,G)$ is a non-permutation diagram.
\end{lem}
\begin{proof}
We have two cases to consider. Firstly consider the diagram $e \in P_n(\delta , G)$ which consists solely of isolated vertices. Let $d^{\prime}$ be the diagram with components $\lbrace 1,\ol{1},\ol{2},\dotsc,\ol{n}\rbrace$ with the trivial colouring and singletons $\lbrace i\rbrace$ for $2\leqslant i \leqslant n$. Then $d^{\prime}e=e$ and we can have no factors of $\delta$ since every vertex in the right-hand column of $d^{\prime}$ is in the same component as a vertex in the left-hand column.

Now suppose that $d\in K_1\cap \cdots \cap K_n$ has at least component with cardinality two or greater. In particular, $d$ has at least one non-propagating edge in the left-hand column. Let $d^{\prime}$ be the diagram defined as follows.
\begin{itemize}
\item $d^{\prime}$ has a non-propagating edge from $i$ to $j$ with label $g$ if and only if $d$ does.
\item $d^{\prime}$ has an edge from $i$ to $\ol{i}$ with the trivial colouring for all $1\leqslant i\leqslant n$.
\end{itemize}  
The diagram $d^{\prime}$ cannot be a permutation diagram since it contains at least one non-propagating edge. One readily checks that $d^{\prime} d=d$: the definition of $d^{\prime}$ means that the composite preserves the non-propagating edges and isolated vertices in the left-hand column of $d$ without introducing any new edges. Furthermore, we can obtain no factors of $\delta$ because every vertex in the right-hand column of $d^{\prime}$ is the same component as a vertex in the left-hand column.
\end{proof}

\begin{prop}
\label{vanish-prop}
For any $\delta \in k$, there exist isomorphisms of $k$-modules
\[\Tor_n^{P_n(\delta,G)}(\1,k[G\wr \Sigma_n]) =0 \quad \text{and} \quad \Ext_{P_n(\delta,G)}^n(k[G\wr \Sigma_n],\1)=0.\]
\end{prop}
\begin{proof}
By Propositions \ref{hom-over-group-ss-prop} and \ref{cohom-over-group-ss-prop}, we have
\[\Tor_n^{P_n(\delta,G)}(\1,k[G\wr \Sigma_n])\cong \1\otimes_{P_n(\delta,G)} K_1\cap \cdots \cap K_n\]
and
\[\Ext_{P_n(\delta,G)}^n(k[G\wr \Sigma_n],\1)\cong \Hom_{P_n(\delta,G)}(K_1\cap \cdots \cap K_n , \1).\]
The module $\1\otimes_{P_n(\delta,G)} K_1\cap \cdots \cap K_n$ is spanned $k$-linearly by elementary tensors of the form $1\otimes d$ where $d$ is a basis diagram in $K_1\cap \cdots \cap K_n$. However, using Lemma \ref{K-decomp-lem}, we have
\[1\otimes d = 1\otimes d^{\prime}d= 1\cdot d^{\prime}\otimes d= 0,\]
and for any $f\in \Hom_{P_n(\delta,G)}(K_1\cap \cdots \cap K_n,\1)$ we have
\[f(d)=f(d^{\prime}d)=d^{\prime}f(d)=0\]
for all basis diagrams $d\in K_1\cap \cdots \cap K_n$.
\end{proof}

\begin{thm}
\label{iso-thm2}
For any $\delta \in k$, there exist isomorphisms of $k$-modules
\[\Tor_n^{P_n(\delta,G)}(\1,\1) \cong H_{n}(G\wr\Sigma_n,\1) \quad \text{and} \quad \Ext_{P_n(\delta,G)}^n(\1,\1)\cong H^{n}(G\wr\Sigma_n,\1).\]
\end{thm}
\begin{proof}
Consider the map of $k$-algebras $P_n(\delta,G)\rightarrow k[G\wr\Sigma_n]$. Recall from \cite[Theorem 12.1]{McCleary} that there exists a homologically-graded first quadrant change-of-rings spectral sequence of the form: 
\[E_{\alpha,\beta}^2\cong \Tor_{\alpha}^{k[G\wr \Sigma_n]}(\Tor_{\beta}^{P_n(\delta,G)}(\1,k[G\wr \Sigma_n]),\1)\Rightarrow \Tor_{\alpha+\beta}^{P_n(\delta,G)}(\1,\1).\]
By Propositions \ref{hom-over-group-ss-prop} and \ref{vanish-prop}, we have $E_{\alpha,0}^2\cong H_{\alpha}(G\wr\Sigma_n,\1)$, $E_{0,n}^2=0$ and $E_{\alpha,\beta}^2=0$ for $1\leqslant \beta\leqslant n$ from which the homological statement follows.

Similarly, there exists a cohomologically graded first quadrant change-of-rings spectral sequence
\[ E_2^{\alpha,\beta}\cong \Ext_{k[G\wr\Sigma_n]}^{\alpha}(\1, \Ext_{P_n(\delta,G)}^{\beta}(k[G\wr\Sigma_n],\1))\Rightarrow \Ext_{P_n(\delta,G)}^{\alpha+\beta}(\1,\1)\]
and Propositions \ref{cohom-over-group-ss-prop} and \ref{vanish-prop} imply the cohomological statement.
\end{proof}

\end{document}